\documentclass[11pt,letterpaper]{amsart}

\usepackage{custom-template, shortcuts,graphicx}
\makeatletter
\@namedef{subjclassname@2020}{\textup{2020} Mathematics Subject Classification}
\makeatother
\begin{document}
	
	\title[Prime torsion in Tate--Shafarevich groups]{Elements of prime order in Tate--Shafarevich groups of abelian varieties over $\Q$}
	\author{Ari Shnidman and Ariel Weiss}
	\date{}
	
	
	\address{Ari Shnidman, Einstein Institute of Mathematics, The Hebrew University of Jerusalem, Edmund J.\ Safra Campus, Jerusalem 9190401, Israel.\vspace*{-3pt}}
	\email{ariel.shnidman@mail.huji.ac.il}
	\address{Ariel Weiss, Department of Mathematics, Ben-Gurion University of the Negev, Be'er Sheva 8410501, Israel.\vspace*{-3pt}}
	\email{arielweiss@post.bgu.ac.il}
	
	\begin{abstract}
		For each prime $p$, we show that there exist geometrically simple abelian varieties $A$ over $\Q$ with $\Sha(A)[p]\ne 0$. Specifically, for any prime $N\equiv 1 \pmod p$, let $A_f$ be an optimal quotient of $J_0(N)$ with a rational point $P$ of order $p$, and let $B = A_f/\langle P \rangle$.  Then the number of positive integers $d \leq X$ with $\Sha(\widehat B_d)[p] \neq 0$ is $ \gg X/\log X$, where $\widehat B_d$ is the dual of the $d$-th quadratic twist of $B$. We prove this more generally for abelian varieties of $\GL_2$-type with a $p$-isogeny satisfying a mild technical condition. In the special case of elliptic curves, we give stronger results, including many examples where $\Sha(E_d)[p] \neq 0$ for an explicit positive proportion of integers $d$.  
	\end{abstract}
	
	\maketitle
	
	\section{Introduction}
	
	Let $A$ be an abelian variety over $\Q$. The Tate--Shafarevich group of $A$ is the abelian group
	\[\Sha(A) = \ker\left(H^1(\Q, A) \to \prod_{\l \leq \infty} H^1(\Q_\l, A)\right)\] 
	classifying $\Q$-isomorphism classes of $A$-torsors that are isomorphic to $A$ over $\Ql$ for every prime $\l$, including the infinite prime $\ell = \infty$.  Much is conjectured, but little is known about the structure of this group.  Famously, the Birch and Swinnerton-Dyer conjecture predicts that $\Sha(A)$ is finite.  On the other hand, heuristics suggest that, for each prime $p$, a positive proportion of elliptic curves $E/\Q$, ordered by height, have $\Sha(E)[p] \neq 0$ \cites{Delaunay, PoonenRains, BKLPR}, and one expects something similar for abelian varieties of higher dimension.
	
	In stark contrast to these expectations, it seems to be an open question whether, for each prime $p$, there exists even a {\it single} elliptic curve over $\Q$ with $\Sha(E)[p] \neq 0$. There are constructions of elliptic curves and higher dimensional abelian varieties $A$ over number fields $K$ with $\Sha(A)[p] \neq 0$ \cite{Kloosterman, ClarkSharif, Creutz}, though the degree of $K$ grows as a function of $p$ in these results. Taking Weil restrictions of these examples gives abelian varieties $A'/\Q$ with $\Sha(A')[p] \neq 0$.  However, there again seem to be no known examples of {\it geometrically simple} abelian varieties $A/\Q$ with $\Sha(A)[p] \neq 0$, for large primes $p$.  The purpose of this paper is to provide such examples.  
	
	\begin{introtheorem}\label{thm:mainthm}
		For each prime $p$, there exists a geometrically simple abelian variety $A/\Q$ such that $\Sha(A)[p]\ne 0$. 
	\end{introtheorem}
	
	In fact, for each $p$, we exhibit infinitely many such $A/\Q$ in distinct $\overline \Q$-isogeny classes. Our examples arise from optimal quotients $A_f$ of the modular Jacobian $J_0(N)$, attached to weight two newforms $f \in S_2(\Gamma_0(N))$ of prime level $N$.
	
	\begin{introtheorem}\label{thm:primetwists}
		Let $N$ be a prime and let $p \geq 3$ be a prime divisor of $(N-1)/\gcd(12,N-1)$. Let $A_f$ be any optimal  quotient of $J_0(N)$ containing a point $P \in A_f(\Q)$ of order $p$. Let $B = A_f/\langle P \rangle$, let $\phi\:A_f\to B$ be the canonical isogeny, and let $\widehat\phi \colon \widehat B \to \widehat A_f$ be the dual isogeny. Then,
		\[\#\{d : 0 < |d| \leq X  \mbox{ and } \Sha(\widehat B_d)[p] \neq 0\} \gg \dfrac{X}{\log X},\]
		where $\widehat B_d$ is the $d$-th quadratic twist of $\widehat B$, for each $d \in \Z$.
	\end{introtheorem}
	
	\Cref{thm:mainthm} follows immediately from \Cref{thm:primetwists}. Indeed, the abelian varieties $A_f$ and $\widehat B_d$ are geometrically simple \cite[Cor.\ 1.4]{Ribet-semistable} and, for any prime $p$ dividing $\frac{N-1}{(12,N-1)}$, there exists at least one newform $f \in S_2(\Gamma_0(N))$ such that $A_f(\Q)$ contains a point of order $p$ \cite[Thm.\ B]{Emerton}. Hence, given a prime $p$, it suffices to apply \Cref{thm:primetwists} to any  prime $N\equiv 1\pmod p$. Moreover, Dirichlet's theorem on primes in arithmetic progressions guarantees that there are infinitely many such $N$. Since $N$ is prime, $J_0(N)$ is semistable with conductor a power of $N$, and we see that different choices of $N$ give geometrically non-isogenous abelian varieties $A/\Q$ with $\Sha(A)[p]\ne 0$.
	
	The dimensions of the abelian varieties $A$ we produce grow with $p$. Indeed, the Weil conjectures imply that $\dim A_f \geq \frac{\log p}{2\log(1+\sqrt 2)}$ \cite[Prop.\ 7.2]{Mazur-Eisenstein}; in particular, $A_f$ is not an elliptic curve if $p \geq 7$. One could give a very crude {\it upper bound} for the minimal dimension of an $A$ with $\Sha(A)[p] \neq 0$, by combining bounds for the smallest prime $N\equiv 1\pmod p$ with well-known bounds on the dimension of $J_0(N)$. 
	
	Our proof of \Cref{thm:primetwists} is fairly short, but uses several deep inputs. First, we prove that if $\phi\:A\to B$ is an $p$-isogeny of abelian varieties, whose Selmer ratio $c(\phi)$ equals $p^i$ (see \Cref{def:selmer-ratio}), then there exists an explicit, positive density set of squarefree integers $\Sigma^+$, such that $\rk A_d(\Q) + \dim_{\F_p} \Sha(A_d)[p] \ge i$ for all but finitely many quadratic twists $A_d$ with $d\in\Sigma^+$ (\Cref{prop:sha}). Our proof uses techniques from Galois cohomology, in particular, the Greenberg--Wiles formula.
	
	In the case that $A=A_f$ is an optimal quotient of $J_0(N)$, we then invoke non-vanishing results for special values of $L$-functions, due to Bump--Friedberg--Hoffstein and Ono--Skinner \cite{BFH,ono-skinner}, to show that $L(f_d,s) \neq 0$, for many quadratic twists $f_d$ of $f$.  By work of Gross--Zagier and Kolyvagin--Logachev \cite{GrossZagier, KolyvaginLogachev} or Kato \cite{Kato}, we have $\rk \widehat B_d(\Q) = 0$ for such $d$. It follows that, if $A=A_f$ is an optimal quotient of $J_0(N)$ that admits a $p$-isogeny $\phi \colon A \to B$ such that $c(\phi)\ge p^2$, then $\Sha(A_d)[p]\ne 0$ for infinitely many quadratic twists $A_d$ (\Cref{thm:generaltwists}).  
	
	In \Cref{sec:proof}, we specialize to prime $N$, where we use Mazur's study of the Eisenstein ideal \cite{Mazur-Eisenstein, Emerton} to show that the condition on $c(\phi)$ is always satisfied, thereby proving \Cref{thm:primetwists}. Our computation of $c(\phi)$ is a generalization of \cite[\S 6]{shnidmanRM}, which was for the prime $p = 3$. 

	In the case of elliptic curves, we prove an even stronger bound than the one in \Cref{thm:primetwists}, by invoking recent work of Smith \cite{Smith-thesis} instead of results on $L$-functions. In \Cref{sec:elliptic-proof} we prove: 
	
	\begin{introtheorem}\label{thm:elliptic}
		Let $E$ be an elliptic curve over $\Q$ with a degree $p$ isogeny $\phi\:E \to E'$, for some prime $p \geq 3$.  Assume that $c(\phi) \geq p^2$ and that $E[2](\Q) \neq \Z/2\Z$. Then for a positive proportion of squarefree integers $d$, we have $\Sha(E_d)[p] \neq 0$. 
	\end{introtheorem}
	
	The  hypothesis $c(\phi)\ge p^2$ applies to ``most'' quadratic twist families of $p$-isogenies of elliptic curves with $p \ge 3$, in a certain sense (see \Cref{prop:atleast3}).
	
	For each elliptic curve $E$ in \Cref{thm:elliptic}, we can give an explicit lower bound on the proportion of $d$ such that $\Sha(E_d)[p]\ne 0$. Sometimes these bounds are larger than those predicted by a na\"ive generalization of the heuristics of Delaunay \cite{delaunay2} and Poonen--Rains \cite{PoonenRains}. For example, in \Cref{sec:elliptic}, we prove that, for the elliptic curve $E\:y^2 +y = x^3-x^2-7820x-263580$ with LMFDB label 11.a1, we have $\Sha(E_d)[5]\ne 0$ for at least $22.9\%$ of squarefree integers $d$. This example shows that the heuristics of Delaunay and Poonen--Rains on distributions of $p$-Selmer groups need to be modified when applied to quadratic twist families of elliptic curves with a rational $p$-isogeny.  
	
	In \cite[Conj.\ 1]{bkls2}, Bhargava, Klagsbrun, Lemke Oliver, and the first author state a conjecture that, for any abelian variety $A/\Q$, we should have $\Sha(A_d)[p] \neq 0$ for a positive proportion of squarefree integers $d$.  They prove special cases of this conjecture, without invoking Smith's work, when one of $E$ or $E'$ admits an additional $3$-isogeny, in addition to a $p$-isogeny.  In general, most known systematic constructions of elements in $\Sha(A)[p]$ over $\Q$ exploit either multiple isogenies or the Cassels--Tate pairing; see e.g.\ \cite{CasselsVI, Fisher, Flynn, shnidmanRM,BhargavaHoII, BFS}.
	One exception is a theorem of Balog--Ono \cite[Thm.\ 2]{BalogOno}, which applies to a large class of elliptic curves $E/\Q$ with a point of order $p$.  As with our proof of Theorem \ref{thm:primetwists}, their proof relies on non-vanishing results for $L$-functions, but to prove $\Sel_p(E_d) \neq 0$, they instead use non-vanishing results for class groups. This leads to the weaker bound $\{d : 0 < |d| \leq X, \, \Sha(E_d)[p] \neq 0\} \gg X^{\frac12 + \frac{1}{2p}}/\log^2X$.
	Thus, even in the special case of elliptic curves, we improve significantly on the known quantitative results, whenever our method applies.  
	
	\section{Selmer groups of abelian varieties with a $p$-isogeny}

	Let $\phi \colon A \to B$ be an isogeny of abelian varieties over $\Q$.
	
	\subsection{Selmer groups and the Selmer ratio}

	\begin{definition}
	    The $\phi$-Selmer group is 
	\[\Sel(\phi) = \ker\left(H^1(\Q, A[\phi]) \to \prod_{\l \leq \infty} H^1(\Ql, A)\right).\] 
	In the special case $A = B$ and $\phi = [p]_A$, we write $\Sel_p(A)$ instead of $\Sel([p]_A)$. 
	\end{definition}
	
	\begin{definition}\label{def:selmer-ratio}
		For $\l$ a finite or infinite prime, define the \emph{local Selmer ratio}
		\[c_\l(\phi) = \dfrac{\#\coker(A(\Ql) \to B(\Ql))}{\#\ker(A(\Ql) \to B(\Ql))}.\]
		When $\l = \infty$, we use the convention that $\Q_\infty = \R$. We then define the \emph{global Selmer ratio} \[c(\phi) = \prod_{\l \leq \infty} c_\l(\phi).\]
	\end{definition}
	These Selmer ratios were defined in \cite{elk}, but were already studied in \cite{CasselsVIII} under a slightly different guise. The notation is meant to recall the Tamagawa number $c_\ell(X)$ of an abelian variety $X$ over $\Q_\ell$.  Indeed, the following lemma shows that, for all but finitely many primes $\ell$, we have $c_\ell(\phi) = c_\ell(B)/c_\ell(A)$.

	\begin{lemma}\label{lem:schaefer}
		For any finite prime $\ell$, we have
		\[c_\l(\phi) = \frac{c_\l(B)}{c_\l(A)}\gamma_{\phi, \l}, \]
		where $\gamma_{\phi, \l}\ii$ is the normalized absolute value of the determinant of the map $\mathrm{Lie}(\mathcal{A}) \to \mathrm{Lie}(\mathcal{B})$ on tangent spaces of the N\'eron models over $\Z_\ell$. In particular, if $\l\nmid \deg(\phi)$, then $\gamma_{\phi, \l}=1$.
	\end{lemma}
	
	\begin{proof}
	This lemma is \cite{schaefer}*{Lem.\ 3.8}. Recall that $c_\ell(A) = \#A(\Q_\ell)/A_0(\Q_\ell)$, where $A_0(\Q_\ell)$ is the subgroup of points which reduce to the identity component in the special fiber of the N\'eron model $\mathcal{A}$ of $A$.    
	\end{proof}
	
	The local Tamagawa numbers $c_\ell(A), c_\ell(B)$ are equal to 1 for all primes $\ell$ of good reduction. Hence, we have $c_\ell(\phi) = 1$ for all but finitely many primes, so the global Selmer ratio $c(\phi) = \prod_\ell c_\ell(\phi)$ is well-defined. Moreover, if $\phi$ has prime degree $p$, then $c(\phi)$ is an integer power of $p$.

	\subsection{Lower bounds on Selmer groups}
	
	Let $N$ be the radical of the conductor of $A$. Thus, a prime $\ell$ divides $N$ if and only if $A$ has bad reduction at $\ell$. 
	
	Let $\Sigma^+$ be the set of positive squarefree integers $d$, such that $d \in \Zl^{\times2}$ for all primes $\l\mid pN$. For any squarefree $d\in \Z$, write $\phi_d\:A_d\to B_d$ for the $d$-th quadratic twist of $\phi$, which again has degree $p$.
	
	\begin{theorem}\label{prop:sha}
		Suppose that $\phi\colon A\to B$ is a degree $p$ isogeny, and write $c(\phi) = p^i$, for some $i \in \Z$. Then for all but finitely many $d \in \Sigma^+$, we have $\dim_{\Fp}\Sel_p(A_d)\geq i$, and hence 
		\[\rk A_d(\Q) + \dim_{\F_p} \Sha(A_d)[p] \geq i.\]
	\end{theorem}
	
	The proof will require several lemmas. 
	
	\begin{lemma}\label{lem:boringprimes}
		If $\l\nmid pN\infty$, then $c_\l(\phi_d) = 1$ for all non-zero $d\in \Z$.
	\end{lemma}
	
	\begin{proof}
		Let $\chi_d\:\GaQl\to \F_p^{\times}$ denote the character corresponding to the (possibly trivial) extension $\Ql(\sqrt{d})/\Ql$.  If $\chi_d$ is unramified, then $A_d$ has good reduction over $\Ql$, and $c_\l(\phi_d) = c_\l(B_d)/c_\l(A_d) = 1$ by \Cref{lem:schaefer}. Assume now that $\chi_d$ is ramified (and in particular, non-trivial). Since $A$ has good reduction over $\Ql$, the extension $\Ql(A[\phi])/\Ql$ is unramified. Hence, the $\GaQl$-action on $A_d[\phi_d]\simeq A[\phi] \otimes_{\F_p} \chi_d$ is via a non-trivial character $\widetilde\chi_d$. Thus  
		\[c_\l(\phi_d) = \frac{\#\im (B_d(\Ql)\to H^1(\Ql, \widetilde\chi_d))}{\#H^0(\Ql, \widetilde\chi_d)}.\]
		The denominator is 1 since $\widetilde\chi_d$ is non-trivial. Let $\epsilon\:\GaQl\to \F_p^{\times}$ be the mod $p$ cyclotomic character, which is unramified. We have $\# H^1(\Ql, \widetilde\chi_d) = \# H^0(\Ql,\widetilde\chi_d)\#H^0(\Ql, \widetilde\chi_d^{-1}\epsilon) = 1$, by local Tate duality \cite{Milne-ADT}*{Cor.\ I.2.3} and the Euler characteristic formula \cite{Milne-ADT}*{Thm.\ I.2.8}. Hence, $c_\l(\phi_d)=1$.
	\end{proof}
	
	\begin{lemma}\label{lem:infty}
		We have $c_\infty(\phi) = \#A[\phi](\R)\ii$.
	\end{lemma}
	
	\begin{proof}
		We have 
		\[\#\coker(A(\R)\to B(\R)) = \#\im (B(\R)\to H^1(\Gal(\C/\R), A[\phi])).\] 
		Since $\#A[\phi] = p$ is odd, we have $H^1(\Gal(\C/\R), A[\phi]) = 0$. Hence, \[c_\infty(\phi) = \#\coker(A(\R)\to B(\R))\cdot\#A[\phi](\R)\ii=\#A[\phi](\R)\ii.\]
	\end{proof}

	To prove \Cref{prop:sha}, we will use the Greenberg--Wiles formula \cite[Thm.\ 8.7.9]{NSW:cohomologyofnumberfields}, which is a consequence of Poitou--Tate duality. If $\psi \colon X \to Y$ is an isogeny, it relates the size of $\Sel(\psi)$ to the size of $\Sel(\widehat\psi)$, where $\widehat \psi \colon \widehat Y \to \widehat X$ is the dual isogeny. 
	Applied to the isogeny $\phi_d \colon A_d \to B_d$, it reads
	\begin{equation}\label{eq:GreenbergWiles}
		c(\phi_d) =  \dfrac{\#\Sel(\phi_d)}{\#\Sel(\widehat\phi_d)} \cdot \dfrac{\#\widehat B_d[\widehat\phi_d](\Q)}{\# A_d[\phi_d](\Q)}.
	\end{equation}
	
	\begin{proof}[Proof of Theorem $\ref{prop:sha}$]
	    First note that, if $A$ is any abelian variety and $p>2$, then $A_d[p](\Q) = 0$ for all but finitely many quadratic twists of $A$.  Indeed, if $0 \neq P \in A_d[p](\Q)$, then the Galois module $A[p]$ has a subrepresentation isomorphic to the quadratic character $\chi_d \colon \Ga\Q \to \F_p^\times$ that cuts out the extension $\Q(\sqrt d)/\Q$.  Since $A[p]$ is finite-dimensional, there can only be finitely many such $d$. As a consequence, if $\phi\:A\to B$ is a $p$-isogeny, then $A[\phi]\subset A[p]$, and we see that $A_d[\phi_d](\Q) = 0$ for all but finitely many $d$.
	    
	    We may therefore ignore the finitely many $d\in\Sigma^+$ such that $\#A_d[\phi_d](\Q)\#\widehat B_d[\widehat\phi_d](\Q) \ne 1$. Hence, from $(\ref{eq:GreenbergWiles})$, we have
	    \[\#\Sel(\phi_d) = c(\phi_d)\#\Sel(\widehat\phi_d)\ge c(\phi_d).\]
	    Now, if $d\in \Sigma^+$, then $\phi_d = \phi$ over $\Q_\ell$, for all primes $\l\mid pN\infty$. Hence, by \Cref{lem:boringprimes}, we have $c(\phi_d) = c(\phi) =p^i$ for all $d\in \Sigma^+$. It follows that
	    \[\dim_{\Fp}\Sel(\phi_d)\ge  i\]
	    for all but finitely many $d\in \Sigma^+$.
	    
	    Finally, we note that, for all but finitely many $d\in \Sigma^+$, the inclusion $A_d[\phi_d]\to A_d[p]$ induces an injection 
	    \[\Sel(\phi_d)\hookrightarrow\Sel_p(A_d).\]
	    Indeed, by \cite[(9.1)]{bkls}, the kernel of this map is $B_d[\psi_d](\Q)/\phi_d(A_d[p](\Q))$, where $\psi_d \colon B_d \to A_d$ is the isogeny such that $\psi_d \circ \phi_d = [p]$. As before, this kernel vanishes for all but finitely many $d$.  Hence, for such $d$, we have $\dim_{\Fp}\Sel_p(A_d)\ge  i$, and the exact sequence
		\[0\to A_d(\Q)/pA_d(\Q) \to \Sel_p(A_d)\to \Sha(A_d)[p]\to 0,\]
		implies that $\rk A_d(\Q) + \dim_{\F_p} \Sha(A_d)[p] \geq i$.
	\end{proof}

	\section{Quotients of $J_0(N)$ with a $p$-isogeny}\label{sec:general}

	For $N \geq 1$, let $J_0(N)$ be the Jacobian of the modular curve $X_0(N)$ over $\Q$. 
	
	\begin{theorem}\label{thm:generaltwists}
		Let $A$ be a simple abelian variety over $\Q$ arising as a quotient of $J_0(N)$, for some integer $N\ge 1$. Assume that $A$ admits a degree $p$ isogeny $\phi \colon A \to B$ over $\Q$, for some prime $p \geq 3$, and that $c(\phi) \geq p^2$. Then   
		\[\#\{d : 0 < |d| \leq X  \mbox{ and } \Sha(A_d)[p] \neq 0\} \gg_A \dfrac{X}{\log X}.\]
	\end{theorem}

	\begin{proof}
		Let $M\mid N$ be the minimal positive integer such that $A$ is a quotient of $J_0(M)$. Then there is a newform $f = \sum a_nq^n\in S_2(\Gamma_0(M))$, such that the coefficient field $E := \Q(\{a_n\})$ is isomorphic to $\End_\Q(A) \otimes \Q$ and $L(A, s) = \prod_{\sigma\in \Hom(E,\C)} L(f^\sigma,s)$. Since $J_0(M)$ has good reduction at all primes $\l \nmid M$, so does $A$.
		
		Let $\Sigma$ be the set of squarefree integers $d$, such that $d\in \Zl^{\times 2}$ for all primes $\l\mid pM$. By \cite[Cor.\ 3]{ono-skinner}, we have
		\[\#\{d\in \Sigma : 0<|d|\leq X \text{ and } L( f_d,1) \ne 0\} \gg \dfrac{X}{\log X},\]
		where $f_d$ is the $d$-th quadratic twist of $f$.  Moreover, by \cite[Cor.\ 14.3]{Kato}, or alternatively by \cite{GrossZagier, KolyvaginLogachev} and \cite[Thm.\ 1]{BFH}, we have $\rk A_d(\Q) = 0$ whenever $L( f_d,1) \ne 0$.
		
		By \Cref{lem:boringprimes} and by the assumption that $d\in \Zl^{\times 2}$ for all primes $\l\mid pM$, we have  $c_\ell(\phi_d) = c_\ell(\phi)$ for all finite primes $\l$ and all $d \in \Sigma$. If $d>0$, we therefore have $c(\phi_d) = c(\phi) \geq p^2$. If $d< 0$ we have $c_\infty(\phi_d)/c_\infty(\phi) \in \{p, p^{-1}\}$ by \Cref{lem:infty}, and so $c(\phi_d) \geq c(\phi)/p \geq p$. Thus, we have $c(\phi_d) \geq p$ for all $d \in \Sigma$. Applying \Cref{prop:sha} both to $A$ and to $A_{-1}$, we have
		\[\rk A_d(\Q) + \dim_{\Fp}\Sha(A_d) \ge 1\]
		for all but finitely many $d\in \Sigma$. It follows that
		\[\#\{d \in \Sigma : 0 < |d| \leq X  \mbox{ and } \Sha(A_d)[p] \neq 0\} \gg \dfrac{X}{\log X},\]
		as desired.
    \end{proof}
	
	\section{Quotients of $J_0(N)$ with $N$ prime}\label{sec:proof}
	
		Let $N$ be a prime and let $p \geq 3$ be a divisor of $\frac{N-1}{\mathrm{gcd}(12,N-1)}$. Let $\T$ be the finite $\Z$-algebra generated by the Hecke operators acting on the space $S_2(\Gamma_0(N))$ of weight $2$ cusp forms on $\Gamma_0(N)$. For each newform $f\in S_2(\Gamma_0(N))$, let $\lambda_f \colon \T\to\C$ be the homomorphism giving the action of the Hecke operators on $f$, and let $I_f=\ker\lambda_f$. 
	Let $J = J_0(N)$ be the modular Jacobian. Then $\T \hookrightarrow \End_\Q J$, and $A_f := J/I_fJ$ is an abelian variety over $\Q$, called the \emph{optimal quotient} corresponding to $f$ \cite{Emerton}. 
	
	By \cite[Thm.\ 1]{Mazur-Eisenstein}, the torsion subgroup $J_0(N)(\Q)_{\mathrm{tors}}$ is cyclic of order $\frac{N-1}{\gcd(12,N-1)}$, and hence is divisible by $p$. By \cite[Thm.\ B]{Emerton}, there exists at least one optimal quotient $A = A_f$ with a point $P \in A(\Q)$ of order $p$. Let $B = A/\langle P\rangle$ be the quotient, let $\phi\:A\to B$ be the canonical $p$-isogeny over $\Q$, and let $\widehat\phi\:\widehat B\to\widehat A$ be the dual isogeny.
	
	\begin{proof}[Proof of Theorem $\ref{thm:primetwists}$]
		By \Cref{thm:generaltwists}, it is enough to prove that $c(\widehat\phi) = p^2$. By the Greenberg--Wiles formula $(\ref{eq:GreenbergWiles})$, it is equivalent to show that $c(\phi) = p^{-2}$.  By \Cref{lem:boringprimes}, we have $c_\ell(\phi) = 1$ whenever $\l\notin \{p, N, \infty\}$. Moreover, by \Cref{lem:infty}, we have $c_\infty(\phi) = \#\ker(\phi)(\R)\ii = p^{-1}$. To compute the remaining two local Selmer ratios, we use some facts about the N\'eron model of $A$.
		
		\begin{lemma}\label{lem:ell=p}
			We have $c_p(\phi) = 1$.  
		\end{lemma}
		
		\begin{proof}
			Since $A$ has good reduction at $p$, we have $c_p(A) = c_p(B) = 1$.  In the notation of \Cref{lem:schaefer}, we therefore have $c_p(\phi) = \gamma_{\phi, p}$. Now, the generator $P$ of $\ker(\phi)$ is the image of a rational cuspidal divisor under the map $J_0(N) \to A$ by \cite[Thm.\ B]{Emerton}.
			Thus, by \cite[II.11.11]{Mazur-Eisenstein}, $\phi$ extends to an \'etale isogeny of N\'eron models over $\Zp$. It follows that $\mathrm{Lie}(\mathcal{A}) \to \mathrm{Lie}(\mathcal{B})$ is an isomorphism and $\gamma_{\phi, p}=1$.
		\end{proof}
		
		\begin{lemma}\label{lem:N-tam-ratio}
			We have $c_N(\phi) = p^{-1}$.
		\end{lemma}
		
		\begin{proof}
			The Atkin--Lehner operator $W_N$ acts on $A$ by $-1$. Indeed, $A$ belongs to the Eisenstein quotient of $J = J_0(N)$, which is itself a quotient of $J^-$, the maximal quotient of $J$ on which the Atkin--Lehner eigenvalue is $-1$ \cite{Mazur-Eisenstein}*{Prop.\ 17.10}. Hence, the global root number of $f$ is $+1$, and by \cite[Prop.\ 7.1]{ConradStein}, the abelian variety $A = A_f$ has split purely toric reduction.  Moreover, the order $p$ point $P$ reduces to a non-identity component of the special fiber of the N\'eron model of $A$ over $\Z_N$, since specialization induces an isomorphism from $A(\Q)_\mathrm{tors}$ to the component group $\Phi_A$ by \cite[Thm.\ B]{Emerton}. Thus, the lemma follows from \cite[Prop.\ 5.1]{BFS}. 
		\end{proof}
		
		We compute $c(\phi) = \prod_{\l\leq\infty}c_\l(\phi) = c_N(\phi)c_\infty(\phi) = p^{-2}$, as desired.
	\end{proof}
	
	\section{Applications to elliptic curves}\label{sec:elliptic-proof}
	
	\begin{proof}[Proof of Theorem $\ref{thm:elliptic}$]
		By a recent result of Smith \cite{Smith-thesis}*{Cor.\ 1.4}, we have
		\[\lim_{X\to\infty}\frac{\#\set{d : 0 < |d|\le X : \rk E_d(\Q) \leq 1}}{2X} =1,\]
		under the assumption that either $E[2](\Q) = 0$, or $E[2](\Q) = (\Z/2\Z)^2$ and $E$ does not admit a cyclic $4$-isogeny.  In our case, $E$ also admits a $p$-isogeny. If $E[2](\Q) = (\Z/2\Z)^2$, then $E$ cannot also admit a cyclic $4$-isogeny: otherwise, the isogeny class of $E$ would contain a cyclic $8p$-isogeny, but $Y_0(8p)(\Q) = \emptyset$ for $p \geq 3$ \cite[Thm.\ 1]{kenku}. Thus, $E$ satisfies the hypotheses of Smith's result, and by \Cref{prop:sha} and our assumption that $c(\phi)\ge p^2$, it follows that $\Sha(A_d)[p]\ne 0$ for $100\%$ of $d\in \Sigma^+$.
	\end{proof}
	
	It is natural to ask about the scope of \Cref{thm:elliptic}.  In this direction, we have:
	
	\begin{proposition}\label{prop:atleast3}
	Suppose that $p > 2$ and that $\phi \colon E \to E'$ is a $p$-isogeny of elliptic curves over $\Q$. Suppose also that $E$ has at least three primes, distinct from $p$, of multiplicative reduction. Then there exists an integer $d$ such that either $c(\phi_d) \geq p^2$ or $c(\widehat\phi_d) \geq p^2$.	In particular, the conclusion of Theorem $\ref{thm:elliptic}$ applies to at least one of $E$ or $E'$.
	\end{proposition}
	
	\begin{proof}
	    We apply \cite{Dokchitser-local}*{Table 1}. Let $\l_1, \l_2, \l_3$ be the primes of multiplicative reduction, and write $v_i$ for the corresponding $\l_i$-adic valuation. Let $j, j'$ be the $j$-invariants of $E$ and $E'$. Replacing $E, E'$ by quadratic twists, we may assume that all three primes have split multiplicative reduction. Moreover, we may further twist so that, at every other prime, $E$ has either good or additive reduction.
	    
	    After reordering and possibly replacing $\phi$ with its dual, we may assume that $v_i(j) = pv_i(j')$ for $i = 1,2$. Indeed, for each $i$, we either have $v_i(j) = pv_i(j')$ or $v_i(j') = pv_i(j)$, so, possibly replacing $\phi$ with its dual, the first option must happen for at least two primes.
	    
	    Twisting again by an integer $d$, such that $\l_3\mid d$ and $\br{\frac {d}{p}} = 1$ for all other primes of bad reduction, we may assume that $E, E'$ have additive, potentially multiplicative reduction at $\l_3$. Similarly, we may assume that $E[\phi](\R) = 0$. Hence, by \cite{Dokchitser-local}*{Table 1} and by \Cref{lem:schaefer}, we have $c_{\l_i}(\phi) = p$ for $i = 1,2$ and $c_{\l_3}(\phi)=1$. Moreover, as in \Cref{lem:infty}, we have $c_\infty(\phi) = 1$.
	    
	    Twisting at $p$ so that $E$ has additive reduction, we either have $c_{p}(\phi) = 1$ or $p$, again by \cite{Dokchitser-local}*{Table 1}. Finally, by construction, $A$ has good or additive reduction at all the other primes, so $c_\ell(\phi) = 1$ for all primes $\ell \nmid \ell_1\ell_2\ell_3p\infty$ \cite{Dokchitser-local}*{Table 1}. Putting everything together, we have $c(\phi) \ge p^2$.
	\end{proof}
	
    \Cref{prop:atleast3} shows that, in a certain natural sense, \Cref{thm:elliptic} applies to ``most'' twist families of elliptic curves $E/\Q$ with an isogeny of degree $p > 2$. To make this claim more precise, we first recall that for $p \notin \{3,5,7,13\}$, there are only finitely many $j$-invariants of elliptic curves with a $p$-isogeny \cite{Mazur-rational}. The modular curves $X_0(p)$ with $p \in \{3,5,7,13\}$ are all isomorphic to $\P^1$. Hence, there are infinitely many $j$-invariants of such elliptic curves over $\Q$. However, for any $p$ in this set, and for any $k \geq 1$, one can show that $100\%$ of rational points in $X_0(p)(\Q)$, ordered by height, have at least $k$ primes of potentially multiplicative reduction. We will not prove this here, but for arguments along these lines, see \cite{bkls2}.  
\section{An example: an elliptic curve with a $5$-isogeny}\label{sec:elliptic}

In this final section, we look at the example $N = 11$ and $p = 5$. In this case, $J_0(11)$ is an elliptic curve with a $5$-isogeny $\phi \colon J_0(11) \to E$ and we work out an explicit lower bound on the proportion of squarefree integers $d$ with $\Sha(E_d)[5] \neq 0$. 

In \cite{delaunay2}, Delaunay gives a Cohen--Lenstra type heuristic, which predicts that, for fixed $r \in \{0,1\}$, and as $E$ varies over all elliptic curves over $\Q$ ordered by conductor,
\[\Prob(\Sha(E)[p]\ne 0\mid \rk E =r ) = 1 - \prod_{k=1}^\infty (1-p^{1-r-2k}). \]
Assuming Goldfeld's conjecture that $50\%$ of elliptic curves have rank $0$ and $50\%$ have rank $1$, this distribution predicts that
\[\Prob(\Sha(E)[5]\ne 0) =1-\frac12\br{ \prod_{k=1}^\infty (1-5^{1-2k}) +  \prod_{k=1}^\infty (1-5^{-2k})}\approx 0.124132. \]
Delaunay's heuristics were presented for the family of all elliptic curves over $\Q$, however, it is natural to guess that they should hold in quadratic twist families as well, as is suggested by Delaunay \cite{Delaunay}*{Sec.\ 4} and Poonen--Rains \cite{PoonenRains}*{Rem.\ 1.9}. 

In the following example, we prove that the family of quadratic twists of $E$ \emph{do not} follow this distribution. In particular, in twist families of elliptic curves with a $p$-isogeny, our example shows that the distribution of the groups $\Sha(E)[p]$ must follow a different distribution. 

\begin{theorem}\label{thm:example}
Let $E \colon  y^2 + y = x^3 - x^2 - 7820x - 263580$, an elliptic curve of conductor $11$.  Then at least $11/48 \approx 22.9\%$ of squarefree integers $d$ satisfy $\Sha(E_d)[5] \neq 0$.  
\end{theorem}

\begin{proof}
The modular curve $X_0(11)$ is genus 1 and has model $E' \colon y^2 + y = x^3 - x^2 - 10x - 20$. We therefore have $E' \simeq J_0(11) \simeq A_f$, where $f$ is the unique weight two eigenform of level $\Gamma_0(11)$.  The torsion subgroup is order $5$ generated by the point $(5,5)$. The curve $E$ in the theorem is the quotient $\phi \colon E' \to E$ by the subgroup generated by $(5,5)$.  

Since $E[2](\Q) = 0$, by \cite{Smith-thesis}*{Cor.\ 1.4}, we have $\rk E_d(\Q) \le 1$ for $100\%$ of squarefree integers $d$. Hence, by \Cref{prop:sha}, $\Sha(E_d)[5] \neq 0$ whenever $c(\widehat\phi) \ge 5^2$, or equivalently, whenever $c(\phi) \le 5^{-2}$. 
    
    By Lemmas \ref{lem:infty}, \ref{lem:ell=p},  and \ref{lem:N-tam-ratio}, or by \cite{Dokchitser-local}*{Table 1}, we have $c_{11}(\phi) = \frac15$, $c_\infty(\phi) = \frac15$ and $c_5(\phi) = 1$.  Moreover, by \cite{Dokchitser-local}*{Table 1}, we have $c_5(\phi_d) = 1$ for all $d$. If $\ell \nmid 5\cdot 11\cdot \infty$, we have $c_\ell(\phi_d) = 1$ by \Cref{lem:boringprimes}.  Hence, if $\Sigma^+$ denotes the set of positive squarefree integers, such that $d\in \Z_{11}^{\times 2}$, then $c(\phi_d) \le 5^{-2}$ for all $d\in \Sigma^+$.
    
    As a subset of the set of squarefree integers, $\Sigma^+$ has relative density
    \[\frac12 \cdot \frac{5\cdot 11}{11^2-1} = \frac{11}{48}.\]
    Hence, at least $11/48$ of squarefree integers $d$ satisfy $\Sha(E_d)[5] \neq 0$.  
\end{proof}	
	
	The fact that Delaunay's heuristics for the distribution of the groups $\Sha(E_d)[p]$ should be modified in certain cases was already understood in \cite{bkls2}, which is one reason why the authors formulate \cite[Conj.\ 1]{bkls2} without specifying a conjectural proportion.  That paper contains several results which show that Delaunay's distribution does not always hold when the (isogeny class of the) elliptic curve admits at least two independent cyclic isogenies. Theorem \ref{thm:elliptic} shows that this phenomenon persists even in the presence of a single cyclic isogeny. It would be interesting to develop consistent heuristics which describe a conjectural distribution in all cases.
	
	\section*{Acknowledgments}
	The authors thank Manjul Bhargava, Brendan Creutz, and Robert Lemke Oliver for helpful conversations. They also thank the referees for their careful reading and helpful suggestions to improve the exposition. The first author was supported by the Israel Science Foundation (grant No.\ 2301/20). The second author was supported by an Emily Erskine Endowment Fund postdoctoral fellowship at The Hebrew University of Jerusalem, by the Israel Science Foundation (grant No.\ 1953/20), and by the Binational Science Foundation (grant No.\ 2018250).
	
	\bibliography{bibliography}
	\bibliographystyle{alpha} 
	
\end{document}